\documentclass[11pt]{article}
\usepackage{geometry} 
\usepackage{amsmath,amsthm,amssymb}                
\usepackage{graphicx, color}
\usepackage{amssymb}
\usepackage{epstopdf}
\usepackage{pdfsync}
\usepackage{mathabx}

\newtheorem{definition}{Definition}[section]
\newtheorem{lemma}[definition]{Lemma}
\newtheorem{theorem}[definition]{Theorem}
\newtheorem{proposition}[definition]{Proposition}
\newtheorem{corollary}[definition]{Corollary}
\newtheorem{remark}[definition]{Remark}
\newtheorem{example}[definition]{Example}

\numberwithin{equation}{section}

\def\e{\varepsilon}

\def\dxy{\,dx dy}
\def\dx{\,dx}
\def\dy{\,dy}
\def\wto{\rightharpoonup}

\newcommand{\EEE}{\color{black}}

\newcommand\blfootnote[1]{%
  \begingroup
  \renewcommand\thefootnote{}\footnote{#1}%
  \addtocounter{footnote}{-1}%
  \endgroup
}

\title{
Validity and failure of the integral representation \\
of $\Gamma$-limits of convex non-local functionals
}
\author{Andrea Braides
\ and Gianni Dal Maso
\\ 
SISSA, via Bonomea 265, Trieste, Italy}
\date{}
\begin{document} 

\maketitle

\begin{abstract} We prove an integral-representation result for limits of  non-local  quadratic forms on $H^1_0(\Omega)$, with $\Omega$ a bounded open subset of $\mathbb R^d$, extending  the representation on $C^\infty_c(\Omega)$ given by the Beurling-Deny formula in the theory of Dirichlet forms. We give a counterexample showing that a corresponding representation may not hold if we consider analogous functionals in $W^{1,p}_0(\Omega)$, with $p\neq 2$ and $1<p\le d$.



\end{abstract}

\blfootnote{Preprint SISSA  06/2023/MATE}

\section{Introduction}
In this paper we continue our investigation on functionals defined on Sobolev spaces in which a   non-local  part,  in the form of a double integral, is present  beside a usual local part depending on the gradient. In general this question can be formulated as the characterization of limits of functionals of the form
$$
F_k(u)=\int_{\Omega\times\Omega}f_k(u(x)-u(y))d\mu_k(x,y)+\int_\Omega g_k(  x, \nabla u(x)) \,dx,
$$
defined in some Sobolev space $W^{1,p}_0(\Omega)$ with $p>1$. Different types of stability of such a class can be studied: in \cite{BDM-1} we have given a notion of convergence on measures $\mu_k$ that guarantees the separate stability of the integrals on $\Omega\times\Omega$   and on $\Omega$, while in \cite{BDM-2} we have explored conditions under which a limit of a sequence of such functionals may still be of this form, but the integrands of the limit    are determined  by   the  interaction between the local and non-local terms. The theory of Dirichlet forms \cite{Fu} gives the stability of such a class under the only condition that $f_k$ and $g_k$ be quadratic,  with very mild conditions on   the  measures  $\mu_k$ (see \cite{Mosco}). 

In this paper we analyze the properties of the $\Gamma$-limits of these quadratic functionals, and show that, rather surprisingly, the same stability property does not hold if we consider only a slight variation of quadratic forms; namely, when all integrands are $p$-th powers with $p\neq 2$.   Previous examples of lack of stability were known in the case of  relaxation results for non-convex functionals $F$,   where the local term is dropped (or, equivalently, if and $f_k=f$ is not convex  and $g_k\equiv 0$). In that case, the lower-semicontinuous envelope in the weak $L^p$ topology of a double integral can be non-representable in the same form \cite{B,KZ,MCT}. Our counterexample shows that a similar issue arises even for sequences of convex and equicoercive functionals  in $W^{1,p}_0(\Omega)$.

In this paper, we first show a result connected to the theory of Dirichlet forms.
If we consider quadratic forms defined in $H^1_0(\Omega)$ of the type 
\begin{equation}\label{new}
\int_{\Omega\times\Omega}|u(x)-u(y)|^2 a_k(x,y)\dxy+\int_\Omega|\nabla u(x)|^2 \,dx,
\end{equation}
with $\Omega$ a bounded open set in $\mathbb R^d$ and $a_k$ positive functions equibounded in $L^1(\Omega\times\Omega)$ and not concentrating on the boundary,  we prove that  the corresponding $\Gamma$-limit,  in the weak topology  of  $H^1_0(\Omega)$,  can be written on functions $u\in C^\infty_c(\Omega)$ in the form
$$
\int_{\Omega\times\Omega}|u(x)-u(y)|^2d\mu(x,y)+\int_\Omega|\nabla u(x)|^2 \,dx,
$$
where $\mu$ is a positive bounded Radon measure (see Theorems \ref{Integral rep} and \ref{stimeme}, and Proposition \ref{vanish}). The main effort is spent in proving the boundedness of such a measure, which does not seem to follow directly from the representation obtained from the Beurling-Deny formula \cite[Theorem 4.5.2]{Fu}.

%

We will show that if $p\neq 2$ (and $1<p\le d$) this does not hold. More precisely, a counterexample can be obtained as follows. Given $x_0\in\Omega$ and a sequence $\e_k$ of positive numbers converging to $0$,  we set
$$
a_k(x,y)= \begin{cases}
|B_{\e_k}(x_0)|^{-1} & \hbox{ if } y\in B_{\e_k}(x_0)\\
0 & \hbox{otherwise,}\end{cases}
$$
where  $B_r(x)$ denotes  the ball of centre $x$ and radius $r$.
In this case the functionals
\begin{equation}\label{fkpi}
F_k(u)=\int_{\Omega\times\Omega}|u(x)-u(y)|^p a_k(x,y)\dxy+\int_\Omega|\nabla u(x)|^p \,dx,
\end{equation}
defined for $u\in W^{1,p}_0(\Omega)$, have a $\Gamma$-limit, in the weak topology of $W^{1,p}_0(\Omega)$, that can be directly expressed on that space as
\begin{equation}\label{effe-p}
F(u)=\int_\Omega| u(x)- m_p(u)|^p\dx+\int_\Omega|\nabla u(x)|^p \,dx,
\end{equation}
where $m_p(u)$ is the unique minimizer of $t\mapsto \int_\Omega |u(x)-t|^p\dx$.
In the proof of this result the inequality $p\le d$ is crucial, implying that a single point has zero $p$-capacity; this also explains the fact that the $\Gamma$-limit is independent of $x_0$.

If $p\neq 2$ we will show that there exist no continuous function $f\colon
\mathbb R^2\to \mathbb R$ and no non-negative bounded Radon measure $\mu$ on $\Omega\times\Omega$ such that
\begin{equation}\label{effe-p-mu}
\int_\Omega| u(x)- m_p(u)|^p\dx= \int_{\Omega\times\Omega}f(u(x),u(y))d\mu(x,y)
\end{equation}
for all $u\in C^\infty_c(\Omega)$ (see Corollary \ref{cor: main}).

Note that the representation of $F$ as above is not in contrast with the representability as a double integral when $p=2$. Indeed, in that case $$m_2(u)={1\over|\Omega|}\int_\Omega u(y)\dy,$$ so that 
\begin{eqnarray}
\int_\Omega| u(x)- m_2(u)|^2\dx&=&\int_\Omega| u(x)|^2\dx-{1\over|\Omega|}\Bigl(\int_\Omega u(y)\dy\Bigr)^2\nonumber\\
&=&{1\over 2|\Omega|}\int_{\Omega\times\Omega}|u(x)-u(y)|^2 \dxy,\label{intro10}
\end{eqnarray}
and $\mu$ is just a multiple of the Lebesgue measure on $\Omega\times\Omega$.

The same observations lead to an example of failure of integral representability in the theory of relaxation.
This can be obtained by considering the functional defined on $C^1_c(\Omega)$    by
$$
F_{x_0}(u)=\int_\Omega| u(x)- u(x_0)|^p\dx+\int_\Omega|\nabla u(x)|^p \,dx,
$$
where $x_0$ is a given point in $\Omega$. Then the lower-semicontinuous envelope with respect to the weak topology of $ W^{1,p}_0(\Omega)$ is given by the same $F$ as in \eqref{effe-p}, so that it cannot be represented in an integral form. Note that the first term in $F_{x_0}(u)$ can be interpreted as an integral on $\Omega\times\Omega$ with respect to the $d$-dimensional Hausdorff measure restricted to $\Omega\times\{x_0\}$, which is the weak limit of the measures $\mu_k= a_k\dxy$ defined above.

\smallskip
The plan of the paper is as follows. Section \ref{Sec2} is dedicated to the quadratic case. We first apply the Beurling-Deny formula to obtain a representation  on $ C^\infty_c(\Omega)$ of the $\Gamma$-limit   $F$  of  the  functionals $F_k$  in \eqref{new}   involving two measures $\mu$ and $\nu$ on $\Omega\times\Omega$ and $\Omega$, respectively (Theorem \ref{Integral rep}). We then analyze some properties of such measures deriving from the estimates satisfied by $F$, proving that both measures are capacitary and finite (Theorems   \ref{stimeme} and \ref{capacity}). Using some additional lower-semicontinuity and truncation properties, satisfied by the $\Gamma$-limit, we then extend the integral representation to the whole  of  $H^1_0(\Omega)$ (Corollary \ref{cor}). 
Section \ref{Sec3} is devoted to the counterexample described above. We show that the $\Gamma$-limit of functionals \eqref{fkpi} is given by \eqref{effe-p}. We then extend \eqref{effe-p-mu} to characteristic functions $u=1_A$ and show that, in this case, the right-hand and left-hand sides of this equality  depend only on the measure of $A$. A careful inspection of the form of this dependence shows that they must be different if $p\neq 2$, concluding the counterexample (Corollary~\ref{cor: main}).

\section{The case $p=2$}\label{Sec2} Throughout the paper $\Omega$ is a connected bounded open subset of $\mathbb R^d$, with $d\ge 1$, even though some limit arguments become trivial if $d=1$. Let $a_k\in L^1(\Omega\times\Omega)$ be non-negative  functions such that 
\begin{equation}\label{estimaak}
\|a_k\|_{L^1(\Omega\times\Omega)}\le M\hbox{ for all } k
\end{equation}
for some $M>0$, and consider  the functionals 
\begin{equation}\label{efffek}
F_k(u):=\int_{\Omega\times\Omega}|u(x)-u(y)|^2 a_k(x,y)\dxy+\int_\Omega|\nabla u(x)|^2 \,dx
\end{equation}
defined for $u\in H^1_0(\Omega)$. 
 Since they are equicoercive in the weak topology of $H^1_0(\Omega)$, we can use the the sequential characterization of $\Gamma$-limits in the weak topology given in \cite[Proposition 8.10\EEE]{DM}.

Note that each $F_k$ satisfies the following {\em truncation property}: $F_k(\Psi(u))\le F_k(u)$ for every $1$-Lipschitz function $ \Psi\colon \mathbb R\to\mathbb R$  with $\Psi(0)=0$ and for every $u\in H^1_0(\Omega)$.
Moreover, \eqref{estimaak} implies
\begin{equation}\label{osc}
F_k(u)\le  M\,({\rm osc}_\Omega u )^2+\int_\Omega |\nabla u(x)|^2\dx,
\end{equation}
where ${\rm osc}_\Omega u:={\rm ess\,sup}_\Omega u-{\rm ess\,inf}_\Omega u$ denotes the oscillation of $u$ on $\Omega$.

\begin{proposition}\label{prop-BD} Assume that $F_k$ $\Gamma$-converges in the weak topology of $H^1_0(\Omega)$ to a functional $F$. Then $F$ satisfies the following properties:

{\rm (a)} the domain of $F$, $D(F):=\{u\in H^1_0(\Omega): F(u)<+\infty\}$, is a linear space containing $H^1_0(\Omega)\cap L^\infty(\Omega)$;

{\rm (b)} $F$ is a quadratic form; that is, there exists a bilinear form $B\colon D(F)\times D(F)\to \mathbb R$ such that
$F(u)= B(u,u)$ for every $u\in D(F)$;

{\rm (c)} the space $D(F)$ endowed with the norm $\|\cdot\|_F$ defined as
\begin{equation}\label{normF}
\|u\|_F=\bigl(\|u\|^2_{L^2(\Omega)}+ F(u)\bigr)^{1/2}
\end{equation}
is a Hilbert space;

{\rm (d)}
we have
$F(\Psi(u))\le F(u)$  for every $1$-Lipschitz function $\Psi\colon \mathbb R\to\mathbb R$ with $\Psi(0)=0$ and for every $u\in H^1_0(\Omega)$;

{\rm (e)}  the space $C^\infty_c(\Omega)$ is dense in $D(F)$  with respect to the norm $\|\cdot\|_F$.
\end{proposition}

\begin{proof} (a) If $u\in H^1_0(\Omega)\cap L^\infty(\Omega)$ then \eqref{osc} gives
\begin{equation}\label{estimate-0}
F(u)\le \liminf_{k\to+\infty} F_k(u)\le\bigl(2\|u\|_{L^\infty(\Omega)} \bigr)^2M+\|\nabla u\|_{L^2(\Omega;\mathbb R^d)}^2;
\end{equation}
(b) follows from general properties of $\Gamma$-convergence (see \cite[Theorem 11.10]{DM}); (c)  following from the lower semicontinuity of $F$ by a standard argument; (d) can be obtained from the truncation property of $F_k$ using the definition of $\Gamma$-limit.

As for (e), we first prove that $H^1_0(\Omega)\cap L^\infty(\Omega)$ is contained in the closure of $C^\infty_c(\Omega)$ with respect to the norm $\|\cdot\|_F$. Indeed, for every $u\in H^1_0(\Omega)\cap L^\infty(\Omega)$ there exist $u_k\in C^\infty_c(\Omega)$ converging to $u$ in $H^1_0(\Omega)$ and with $\|u_k\|_{L^\infty(\Omega)}\le \|u\|_{L^\infty(\Omega)}$. From this and \eqref{estimate-0} we deduce that $\|u_k\|_F$ is equibounded, which implies that $u_k\wto u$ weakly in the Hilbert space $D(F)$, showing the desired inclusion.
It remain to prove that $H^1_0(\Omega)\cap L^\infty(\Omega)$ is dense in $D(F)$. To that end, if $u\in D(F)$ we can consider $u_m= \Psi_m(u)$, where $\Psi_m(t)= (m\wedge t)\vee (-m)$ is a truncation operator. Then 
$u_m\in H^1_0(\Omega)\cap L^\infty(\Omega)$, $u_m\to u$ in $H^1_0(\Omega)$, and $F(u_m)\le F(u)$
by (d). This again implies that $\|u_m\|_F$ is equibounded, so that $u_m$ weakly converges to $u$ in $D(F)$, concluding the proof.
\end{proof}

\begin{theorem}\label{Integral rep} Assume that $F_k$ $\Gamma$-converges in the weak topology of $H^1_0(\Omega)$ to a functional $F$. Then there exist two positive Radon measures $\mu$ and $\nu$  on $\Omega\times\Omega$ and  $\Omega$, respectively, such that

{\rm(a)} for $u\in C^\infty_c(\Omega)$
\begin{equation}\label{repr-F}
F(u)=\int_{\Omega\times\Omega}|u(x)-u(y)|^2d\mu(x,y)+\int_\Omega|u(x)|^2 d\nu(x)+\int_\Omega|\nabla u(x)|^2 \,dx;
\end{equation}

{\rm(b)} 
$\mu$ is symmetric; i.e., $\mu(A\times B)=\mu(B\times A)$ for every pair of Borel sets $A$ and $B$ contained in $\Omega$;

{\rm(c)} setting $\Delta:=\{(x,x): x\in\mathbb R^d\}$, we have 
\begin{equation}\label{stimadiag}
\mu((\Omega\times\Omega)\cap\Delta)=0.
\end{equation}

\end{theorem}

\begin{proof} From the previous proposition it follows that the bilinear form $B$  defined therein  is a {\em Dirichlet form} (see \cite{Fu}), and that $C^\infty_c(\Omega)$ is a {\em core} for $B$. Consequently by the Beurling-Deny Theorem (\cite[Theorem 2.2.2]{Fu}), we have the decomposition
$$
F(u)= F^n(u)+F^{\ell}(u)+  F^c(u) \hbox{ for every $u\in C^\infty_c(\Omega)$},
$$
where $F^{\ell}$ is a local term, $F^n$ is a non-local term, and  $F^c$ is a local term depending only on the derivatives. More precisely, there exist a symmetric matrix of Radon measures $\mu_{ij}$, and two positive Radon measures $\mu$ and $\nu$ such that
\begin{eqnarray}\label{BD-1}
&&F^n(u)=\int_{\Omega\times\Omega} |u(x)-u(y)|^2d\mu(x,y),
\quad F^{\ell}(u)=\int_\Omega|u(x)|^2 d\nu(x),\\ \label{BD-1}
 &&F^c(u)=\sum_{i,j=1}^d\int_\Omega{\partial u(x)\over\partial x_i}{\partial u(x)\over\partial x_j}d\mu_{ij}(x)
\end{eqnarray}
for every $u\in C^\infty_c(\Omega)$. Note that it is not restrictive to suppose that $\mu$ is symmetric and 
\eqref{stimadiag} holds since $|u(x)-u(y)|^2=0$ on $\Delta$.

We now show that in our case  $\displaystyle F^c(u)=\int_\Omega|\nabla u(x)|^2\dx$.
Note that
\begin{eqnarray}\label{stimatre}
&&\int_\Omega |\nabla u(x)|^2\dx\le F(u)\le  M\,({\rm osc}_\Omega u)^2+\int_\Omega |\nabla u(x)|^2\dx\end{eqnarray}
for every $u\in C^\infty_c(\Omega)$,
using the lower semicontinuity of the first integral for the first inequality and \eqref{osc} for the second one.

Given $\omega\in C^\infty_c(\Omega)$ and $\xi\in\mathbb R^d$, let $\varphi$ and $\psi$ be defined by
$$
\varphi (x)=\omega(x)\cos(x\cdot\xi) \ \hbox{ and }\ \psi (x)=\omega(x)\sin(x\cdot\xi).$$
By a direct computation we have
\begin{eqnarray*}
&&F^c(\varphi )+ F^c(\psi )= F^c(\omega)+\sum_{i,j=1}^d\int_\Omega \omega^2(x)\xi_i\xi_jd\mu_{ij}(x),
\\
&&\int_\Omega|\nabla\varphi (x)|^2\dx+ \int_\Omega|\nabla\psi (x)|^2\dx= \int_\Omega|\nabla\omega(x)|^2\dx+\int_\Omega \omega^2(x)|\xi|^2dx,
\end{eqnarray*}
so that
%
\begin{eqnarray*}
&&\Big|\sum_{i,j=1}^d\int_\Omega \omega^2(x)\xi_i\xi_jd\mu_{ij}(x)-\int_\Omega \omega^2(x)|\xi|^2dx\Big|\\
&\le&\Big|F(\varphi )-\int_\Omega|\nabla\varphi (x)|^2\dx\Big|+\Big|F(\psi )-\int_\Omega|\nabla\psi (x)|^2\dx\Big|+\Big|F(\omega)-\int_\Omega|\nabla\omega(x)|^2\dx\Big|\\
&&+ F^\ell(\varphi )+ F^\ell(\psi )+F^\ell(\omega)+ F^n(\varphi )+F^n(\psi )+F^n(\omega)
\end{eqnarray*}

Using \eqref{stimatre} we get
\begin{eqnarray}\label{stima4}\nonumber
&&\Big|\sum_{i,j=1}^d\int_\Omega \omega^2(x)\xi_i\xi_jd\mu_{ij}(x)-\int_\Omega \omega^2(x)|\xi|^2dx\Big|\\
&&\le12 M \|\omega\|^2_{L^\infty(\Omega)} +3 F^\ell(\omega)+ F^n(\varphi )+F^n(\psi )+F^n(\omega).
\end{eqnarray}
Let $K={\rm supp}\,\omega$. We observe that
\begin{eqnarray*}
F^n(\varphi )&=&\int_{K\times K} |\omega(x)\cos(x\cdot\xi) - \omega(y)\cos(y\cdot\xi)|^2d\mu(x,y)\\
&& +\int_{K\times(\Omega\setminus K)} |\omega(x)\cos(x\cdot\xi)|^2d\mu(x,y)
 +\int_{(\Omega\setminus K)\times K} |\omega(y)\cos(y\cdot\xi)|^2d\mu(x,y)\\
 &\le & 4\|\omega\|^2_{L^\infty(\Omega)}\mu(K\times K)+\int_{K\times(\Omega\setminus K)} |\omega(x)|^2d\mu(x,y)
 +\int_{(\Omega\setminus K)\times K} |\omega(y)|^2d\mu(x,y)
 \\
 &\le & 4\|\omega\|^2_{L^\infty(\Omega)}\mu(K\times K)+\int_{\Omega\times\Omega} |\omega(x)-\omega(y)|^2d\mu(x,y)
  \\
 &\le& 4\|\omega\|^2_{L^\infty(\Omega)}\big(\mu(K\times K)+M\big) +\int_{\Omega}|\nabla\omega(x)|^2\dx,
\end{eqnarray*}
where in the last inequality we again use \eqref{stimatre}. Using the analogous estimate for $F^n(\psi )$, from 
\eqref{stimatre} and \eqref{stima4} we get
\begin{eqnarray*}\nonumber
&&\hskip-1cm\bigg|\sum_{i,j=1}^d\int_\Omega \omega^2(x)\xi_i\xi_jd\mu_{ij}(x)-\int_\Omega \omega^2(x)|\xi|^2dx\bigg|\\
&& \le \|\omega\|^2_{L^\infty(\Omega)}(8\mu(K\times K) +20 M)+5\int_{\Omega}|\nabla\omega(x)|^2\dx.
\end{eqnarray*}
Applying this estimate with $\xi$ replaced by $\lambda\xi$ we obtain
\begin{eqnarray*}\nonumber
&&\hskip-2cm\bigg|\sum_{i,j=1}^d\int_\Omega \omega^2(x)\xi_i\xi_jd\mu_{ij}(x)-\int_\Omega \omega^2(x)|\xi|^2dx\bigg|\\
 &\le& {1\over \lambda^2}\Big(\|\omega\|^2_{L^\infty(\Omega)}(8\mu(K\times K) +20 M)+5\int_{\Omega}|\nabla\omega(x)|^2\dx\Big).
\end{eqnarray*}
Letting $\lambda\to+\infty$ we then have 
$$
\sum_{i,j=1}^d\int_\Omega \omega^2(x)\xi_i\xi_jd\mu_{ij}(x)=\int_\Omega \omega^2(x)|\xi|^2\dx.
$$
By polarization we have 
$$
\sum_{i,j=1}^d\int_\Omega \omega^2(x)\xi_i\eta_jd\mu_{ij}(x)=\int_\Omega \omega^2(x)\xi\cdot\eta\dx \qquad\hbox{ for all }\xi, \eta\in \mathbb R^d.
$$
Taking $\xi,\eta\in\{e_1,\ldots, e_d\}$ and using the arbitrariness of $\omega$, we obtain that $\mu_{ij}=0$ for $i\neq j$ and $\mu_{ii}=\mathcal L^d$, so that $F^c(u)=\int_\Omega|\nabla u(x)|^2\dx$. Using this we also have
\begin{equation}\label{stima6}
F^n(u)+F^\ell(u)\le M\,({\rm osc}_\Omega(u))^2 
\end{equation} 
for all $u\in C^\infty_c(\Omega)$. \end{proof}

The following example shows that  a non-trivial measure measure $\nu$ may indeed appear in the limit.

\begin{example}\rm
Let $d=2$ and $\Omega=(0,1)\times (0,1)$, with $a_k(x,y)= \alpha_k(x)+\alpha_k(y)$, where 
$$
\alpha_k(x)=\begin{cases} 
k &\hbox{ if }x\in R_k:=(0,1)\times(0,{1\over k}),\\
0 &\hbox{ otherwise.}
\end{cases}
$$
Then \eqref{estimaak} is satisfied with $M=2$. We now show that the corresponding $F_k$ converge to the functional given by
$$
F(u)=2\int_\Omega|u(x)|^2\dx+\int_\Omega|\nabla u(x)|^2\dx,
$$
which corresponds to $\mu=0$ and $\nu=2{\cal L}^d$.

In order to prove the liminf inequality we fix $u_k$ converging weakly to $u$ in $H^1_0(\Omega)$.
We have
\begin{eqnarray}\label{coco}\nonumber
&&\int_{\Omega\times\Omega}|u_k(x)-u_k(y)|^2a_k(x,y)\dxy\\
&&=2k\int_{R_k}|u_k(x)|^2\dx
-4k\int_{R_k}u_k(x)\dx\int_\Omega
u_k(y)\dy+ 2\int_\Omega|u_k(y)|^2\dy.
\end{eqnarray}
Moreover, using a Poincar\'e-inequality argument in $R_k$, we obtain
\begin{equation}\label{poin}
k\int_{R_k}|u_k(x)|^2\dx\le {1\over k} \int_{R_k}|\nabla u_k(x)|^2\dx.
\end{equation}
This gives
\begin{equation}\label{limo-0}
\lim_{k\to +\infty} k\int_{R_k}|u_k(x)|^2\dx=0
\end{equation}
since the right-hand side in \eqref{poin} converges to $0$. By H\"older's inequality we also obtain
\begin{equation*}\label{limo}
\lim_{k\to +\infty} k\int_{R_k}|u_k(x)|\dx=0.
\end{equation*}
These limits imply that, by \eqref{coco},
\begin{equation*}\label{limo-2}
\lim_{k\to +\infty}\int_{\Omega\times\Omega}|u_k(x)-u_k(y)|^2a_k(x,y)\dxy=2\int_\Omega|u(x)|^2\dx.
\end{equation*}
By the lower semicontinuity of the gradient term this shows that
$\liminf\limits_{k\to +\infty} F_k(u_k)\ge F(u)$.
On the other hand, \eqref{coco} with $u_k=u$ shows that $\lim\limits_{k\to+\infty} F_k(u)=F(u)$, completing the proof of the $\Gamma$-convergence of $F_k$ to $F$.
\end{example}

We now analyze the properties of the measures $\mu$ and $\nu$ given in Theorem \ref{Integral rep} in order to extend the representation result to the whole $H^1_0(\Omega)$.

\begin{theorem}\label{stimeme} Let $F\colon C^\infty_c(\Omega)\to [0,+\infty)$ be such that there exist exist two positive Radon measures $\mu$ and $\nu$  on $\Omega\times\Omega$ and  $\Omega$, respectively,  
that satisfy {\rm(a)}, {\rm(b)}, and {\rm(c)} of Theorem~{\rm\ref{Integral rep}}.
Suppose in addition that there exists $M>0$ such that
\begin{equation}\label{stima6.1}
F(u)\le  M\,({\rm osc}_\Omega u)^2 +\int_\Omega|\nabla u(x)|^2\dx
\end{equation}
for all $u\in C^\infty_c(\Omega)$.
Then 

{\rm(a)} the measures $\mu$ and $\nu$ are uniquely determined;

{\rm(b)} $\mu(\Omega\times\Omega)<+\infty$ and $\nu(\Omega)<+\infty$.
\end{theorem}
\begin{proof}
We begin by proving that $\nu$ is a finite measure. From \eqref{stima6.1} and \eqref{repr-F} we first obtain
$$
\int_\Omega|u(x)|^2d\nu(x)\le M\,({\rm osc}_\Omega u)^2 \hbox{ for all $u\in C^\infty_c(\Omega)$. }
$$
Approximating the constant $1$ by an increasing sequence of non-negative functions  $u_k\in C^\infty_c(\Omega)$ we obtain that $\nu(\Omega)\le M$. 

\smallskip
We now complete the proof of claim (b), showing that the measure $\mu$ is finite. We preliminarily note that
from \eqref{stima6.1} and \eqref{repr-F} we also obtain
\begin{equation}\label{estimu}
\int_{\Omega\times\Omega}|u(x)-u(y)|^2d\mu(x,y)\le M\,({\rm osc}_\Omega u)^2 \hbox{ for all $u\in C^\infty_c(\Omega)$. }
\end{equation}

Since the proof is rather complex we first consider the case $d=1$, hoping it may clarify the arguments used.
For given $\eta>0$ we let
\begin{equation}\label{diag-eta}
\Delta_\eta:=\{(x,y)\in \mathbb R^d\times  \mathbb R^d: |x-y|\le\eta\},
\end{equation}
and cover $(\Omega\times \Omega)\setminus \Delta_\eta$ `in the average' by a family of `checkerboard-type' sets depending on two parameters $\alpha$ and $\beta$, showing that the covering has some average properties independent of $\eta$.

With given $\alpha, \beta\in\mathbb R$, with $\beta>0$, we define 
\begin{equation}\label{Aab}
A_{\alpha,\beta}=\bigcup_{h\!\ \rm even} [\alpha+h\beta,\alpha+(h+1)\beta),
\end{equation}
and
\begin{equation}\label{Eab}
E_{\alpha,\beta}=\bigcup_{h+k\!\ \rm odd} ([\alpha+h\beta,\alpha+(h+1)\beta)\times [\alpha+k\beta,\alpha+(k+1)\beta)).
\end{equation}
Note that $E_{\alpha+k \beta, \beta}=E_{\alpha, \beta}$ for all $k\in \mathbb Z$ and \begin{equation}\label{ezb}E_{\alpha, \beta}= (\alpha,\alpha)+E_{0,\beta};
\end{equation} moreover, we have 
\begin{equation}\label{equation11}
(\Omega\times\Omega)\cap E_{\alpha,\beta}=
((\Omega\cap A_{\alpha,\beta})\times (\Omega\setminus A_{\alpha,\beta}) )
\cup 
((\Omega\setminus A_{\alpha,\beta})\times (\Omega\cap A_{\alpha,\beta}) ).
\end{equation}

We claim that 
\begin{equation}\label{stima12}
\mu((\Omega\times\Omega)\cap E_{\alpha,\beta})\le M,
\end{equation}
where $M$ is defined in \eqref{stima6.1}. 
To prove the claim we take a sequence $u_k\in C^\infty_c(\Omega)$ such that $0\le u_k\le 1$ and $u_k(x)\to 1_{A_{\alpha,\beta}}(x)$ for all $x\in\Omega$. We then have 
\begin{equation}\label{stima10}
\lim_{k\to+\infty}\int
_{(\Omega\cap A_{\alpha,\beta})\times (\Omega\setminus A_{\alpha,\beta})}
|u_k(x)-u_k(y)|^2d\mu(x,y)= \mu((\Omega\cap A_{\alpha,\beta})\times (\Omega\setminus A_{\alpha,\beta})).
\end{equation}
If this latter measure is finite this limit is obtained by using the Dominated Convergence Theorem; otherwise, it follows by applying Fatou's Lemma.
From \eqref{estimu}, \eqref{stima10}, and the analogous limit for  $\mu((\Omega\setminus A_{\alpha,\beta})\times (\Omega\cap A_{\alpha,\beta}))$, we obtain the claim thanks to \eqref{equation11}.

The next argument is, given $\eta>0$, to determine $\e >0$ such that, setting 
$D_{\e }=\{(\alpha,\beta): 0\le \alpha<\beta,\ \e \le \beta<2\e \}$, we have
  \begin{equation}\label{stima11}
\int_{D_{\e }}1_{E_{\alpha,\beta}}(x,y)d\alpha d\beta\ge \Big({1\over 2}-\eta\Big)|D_{\e }|
\end{equation}
for every $(x,y)\in(\Omega\times\Omega)\setminus\Delta_\eta$, where $\Delta_\eta$ us defined in \eqref{diag-eta}. Once \eqref{stima11} is proved we obtain
\begin{eqnarray*}
\Big({1\over 2}-\eta\Big)|D_{\e }|\mu((\Omega\times\Omega)\setminus\Delta_\eta)
&\le&\int_{\Omega\times\Omega}\int_{D_{\e }}1_{E_{\alpha,\beta}}(x,y)d\alpha d\beta d\mu(x,y)\\
&=&\int_{D_{\e }}\int_{\Omega\times\Omega}1_{E_{\alpha,\beta}}(x,y)d\mu(x,y)d\alpha d\beta\\
&\le&|{D_{\e }}|\mu(E_{\alpha,\beta})\le |{D_{\e }}| M
\end{eqnarray*}
by \eqref{stima12}. Dividing by $|{D_{\e }}|$ we obtain that
\begin{eqnarray*}
\Big({1\over 2}-\eta\Big)\,\mu((\Omega\times\Omega)\setminus\Delta_\eta)\le M.
\end{eqnarray*} Taking into account \eqref{stimadiag}
we obtain 
\begin{equation}
\mu(\Omega\times\Omega)\le 2 M
\end{equation}
by the arbitrariness of $\eta$, concluding the proof of the boundeness of $\mu$. 


It remains to prove \eqref{stima11}. By Fubini's Theorem we have  
$$
\int_{D_{\e }}1_{E_{\alpha,\beta}}(x,y)d\alpha d\beta=\int_{\e }^{2\e } \mathcal L^1(\{\alpha\in [0,\beta) : (x,y)\in E_{\alpha, \beta}\})d\beta.
$$
We now claim that $\mathcal L^1(\{\alpha\in [0,\beta) : (x,y)\in E_{\alpha, \beta}\})$ depends on $z:=y-x$ 
and that, setting
$$
\gamma_z(\beta):=\mathcal L^1(\{\alpha\in [0,\beta) : (x,y)\in E_{\alpha, \beta}\}),
$$we have
$$
\gamma_z(\beta)=|z-2m_\beta(z)\beta|,
$$
where $m_\beta(z)\in\mathbb Z$ is the unique integer such that
$$
(2m_\beta(z)-1)\beta\le z< (2m_\beta(z)+1)\beta.
$$ 

We first observe that, by the periodicity of $E_{\alpha,\beta}$, the set $\{\alpha\in\mathbb R: (x,y)\in E_{\alpha, \beta}\}$ is a periodic subset of $\mathbb R$ of period $\beta$. Moreover, using \eqref{ezb}, we have 
$$
\{\alpha\in\mathbb R: (x,y)\in E_{\alpha, \beta}\}= \{\alpha\in\mathbb R: (0,z)\in E_{\alpha, \beta}\}+\alpha_\beta(x),$$
with $\alpha_\beta(x):=x-\beta\lfloor {x\over\beta}\rfloor$, where $\lfloor\cdot \rfloor$ denotes the integer part; hence,
$$
\mathcal L^1(\{\alpha\in[0,\beta): (x,y)\in E_{\alpha, \beta}\})= \mathcal L^1(\{\alpha\in[0,\beta): (0,z)\in E_{\alpha, \beta}\})
$$
by the periodicity of the set, which proves the fact that $\gamma_z$ indeed depends only on $z$.

 \begin{figure}[h!!]
\centerline{\includegraphics[width=.3\textwidth]{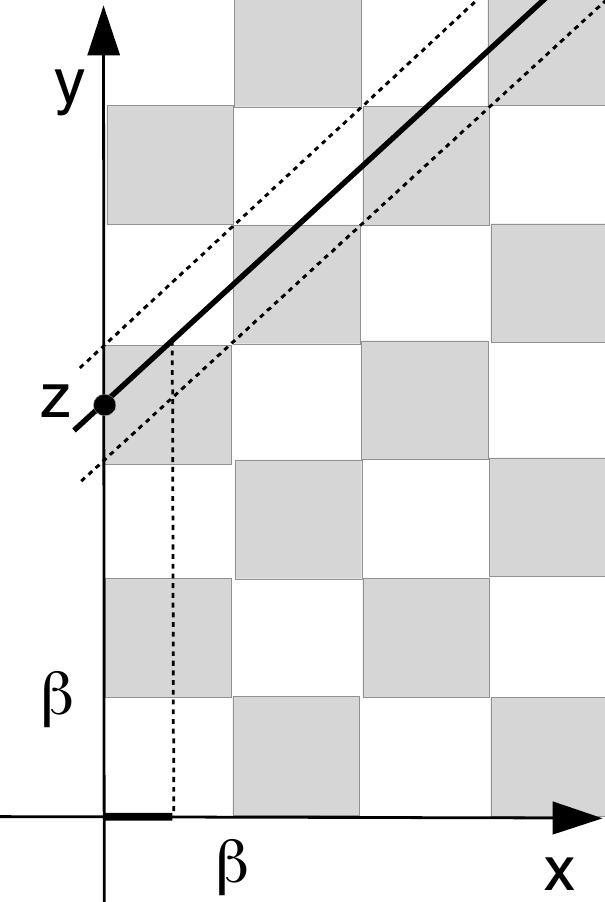}}
\caption{geometrical interpretation of $\gamma_z$; the grey zone represents $E_{0,\beta}$.}
\label{Figure1}
\end{figure}
We now observe that, using the periodicity of $E_{0,\beta}$, we can write
$$
\gamma_z(\beta)= \mathcal L^1(\{\alpha\in[0,\beta): (0,z)\in (\alpha,\alpha)+ E_{0, \beta}\})=
 \mathcal L^1(\{\alpha\in[0,\beta): (0,z)+(\alpha,\alpha)\in  E_{0, \beta}\}),
$$
so that 
we have
$\gamma_z(\beta)=0$ when $|z|= m\beta$ for $m$ even and $\gamma_z(\beta)=\beta$ for  $|z|=\beta m$ for $m$ odd. Otherwise, the function $\gamma_z$ is the piecewise-affine interpolation determined by these conditions on $\beta\mathbb Z$. In Fig.~\ref{Figure1} we give a pictorial representation of the value $\gamma_z(\beta)$, given by the length of the projection on  the $x$-coordinate axis of the intersection of the line $\{(0,z)+(\alpha,\alpha):\alpha\in \mathbb R\}$ with $E_{0, \beta}\cap([0,\beta)\times\mathbb R)$. In the figure we have pictured the cases of a generic $z$ and of the two possibilities, $z$ even or odd, when $z\in\mathbb Z$ (dashed lines).

In terms of $\gamma_z$,  \eqref{stima11} is equivalent to 
  \begin{equation}\label{stima15}
\int_{\e }^{2\e } \gamma_z(\beta)d\beta\ge \Big({1\over 2}-\eta\Big)|D_{\e }|.
\end{equation}
Note that $|D_{\e }|={3\over 2}\e ^2$, but in most of the computations we do not need this explicit value.

 \begin{figure}[h!!]
\centerline{\includegraphics[width=.5\textwidth]{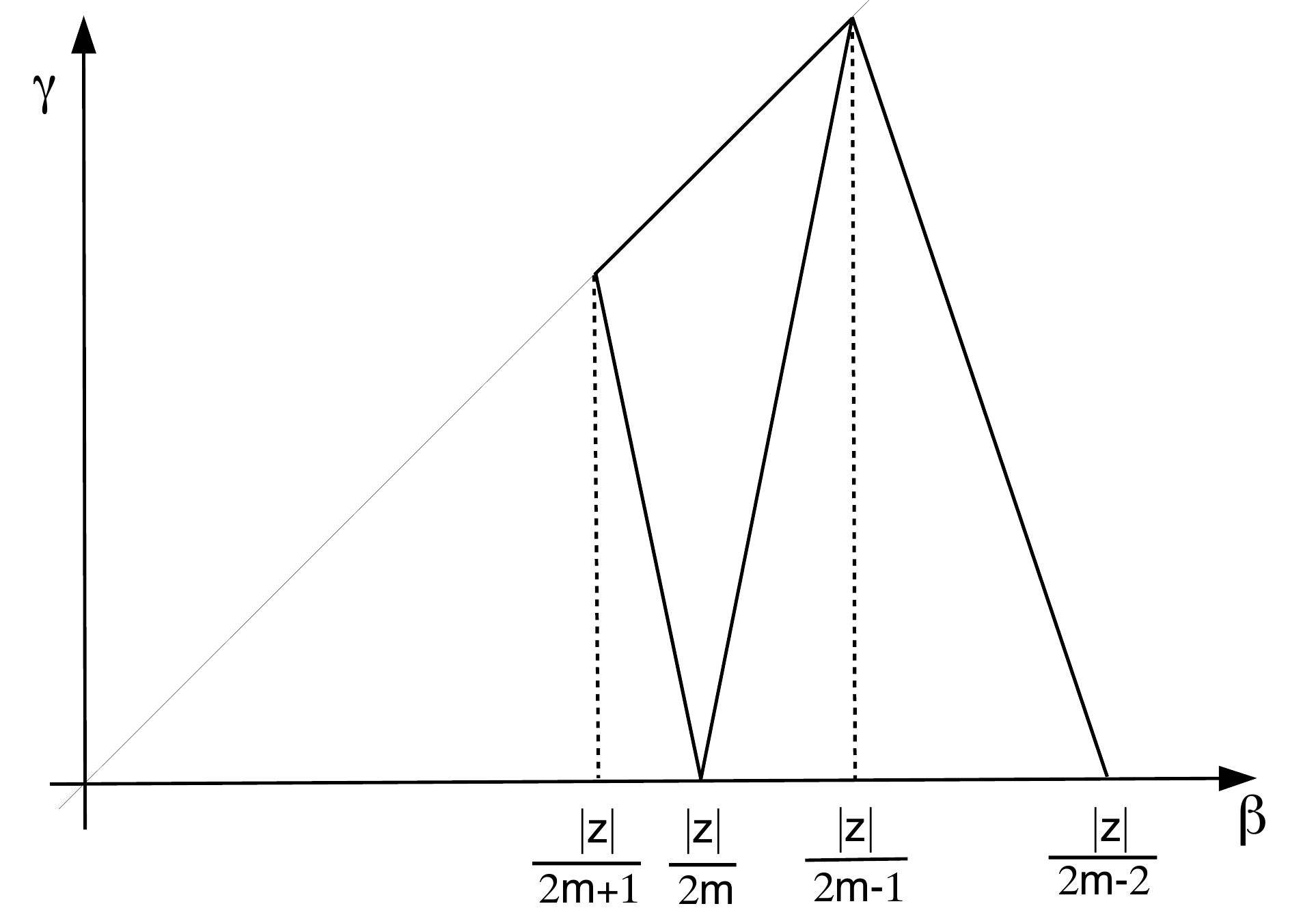}}
\caption{a part of the graph of $\gamma_z$.}
\label{Figure2}
\end{figure}
In order to prove \eqref{stima15} note that
\begin{equation}\label{stima40}
\int_{|z|\over 2m}^{|z|\over 2m-2}\gamma_z(\beta)d\beta
\ge
\int_{|z|\over 2m+1}^{|z|\over 2m-1}(\beta-\gamma_z(\beta))d\beta
\end{equation}
for every $m\in\mathbb N$ with $m\ge 2$ (this is just a comparison between the areas of the two triangles in Fig.~\ref{Figure2}). We define $h=h(|z|,\e )$ and $k=k(|z|,\e )$ setting
$$
h=\max\Big\{m\in\mathbb N: \e \le{|z|\over 2m}\Big\}, \qquad k=\min\Big\{m\in\mathbb N: {|z|\over 2m}<2\e \Big\},
$$
and observe that $k\le h$. From \eqref{stima40}, summing on the set of integer $m$ with $k< m\le h$ we deduce that
\begin{eqnarray*}\nonumber
\int_{\e }^{2\e }\gamma_z(\beta)d\beta&\ge& \int_{|z|\over 2h}^{|z|\over 2k}\gamma_z(\beta)d\beta
\\
&\ge&
\int_{|z|\over 2h+1}^{|z|\over 2k+1}(\beta-\gamma_z(\beta))d\beta\ge \int_{\e }^{2\e }(\beta-\gamma_z(\beta))d\beta-|z|{2\e \over k^2} ,
\end{eqnarray*}
so that 
\begin{equation}\label{stima41}
\int_{\e }^{2\e }\gamma_z(\beta)d\beta\ge {1\over 2}|D_{\e }|-|z|{\e \over k^2}=
|D_{\e }|\Big({1\over 2} -{2|z|\over 3k^2 \e }\Big).
\end{equation}
Note that 
${1\over k}<{4\e \over |z|}$, so that 
${2|z|\over 3k^2 \e }\le {32 \e \over 3 |z|}$.  By choosing $\e <{3\over 32}\eta^2$ we then have ${2|z|\over 3k^2 \e }\le\eta$ for all $z$ with $|z|\ge \eta$, and estimate \eqref{stima15} holds.

\smallskip
If $d>1$, for all $i\in\{1,\ldots, d\}$ and $\alpha, \beta\in\mathbb R$, with $\beta>0$, we consider the set
$$
E^i_{\alpha,\beta}=\{(x,y)\in \mathbb R^d\times \mathbb R^d: (x_i,y_i)\in E_{\alpha,\beta}\},
$$
where $E_{\alpha,\beta}$ is defined in \eqref{Eab}. Correspondingly, we define $$A^i_{\alpha,\beta}=\{(x,y)\in \mathbb R^d\times \mathbb R^d: (x_i,y_i)\in A_{\alpha,\beta}\},
$$
where $A_{\alpha,\beta}$ is defined in \eqref{Aab}.  As in \eqref{equation11}
we get
\begin{equation*}
(\Omega\times\Omega)\cap E^i_{\alpha,\beta}=
((\Omega\cap A^i_{\alpha,\beta})\times (\Omega\setminus A^i_{\alpha,\beta}) )
\cup 
((\Omega\setminus A^i_{\alpha,\beta})\times (\Omega\cap A^i_{\alpha,\beta}) ).
\end{equation*}
Repeating the steps in the proof in the case $d=1$ we obtain, as in \eqref{stima11},
  \begin{equation*}
\int_{D_{\e }}1_{E^i_{\alpha,\beta}}(x,y)d\alpha d\beta\ge \Big({1\over 2}-\eta\Big)|D_{\e }|
\quad\hbox{for all $i\in\{1,\ldots, d\}$,} 
\end{equation*}
and hence that
$\mu((\Omega\times\Omega)\setminus\Delta^i)\le 2M$,
where $\Delta^i=\{(x,y)\in \mathbb R^d\times \mathbb R^d: x_i=y_i\}$.
Since $\Delta=\bigcap_{i=1}^d\Delta^i$, and $\mu((\Omega\times\Omega)\cap\Delta)=0$, we deduce that
\begin{eqnarray*}
\mu(\Omega\times\Omega)\le 2dM,
\end{eqnarray*}
which concludes the proof of (b).

\smallskip
In order to prove (a), we first note that for all disjoint open subsets $A$ and $B$ of $\Omega$, thanks to \eqref{repr-F} and the symmetry of $\mu$ we have
\begin{eqnarray*}
&&2\mu(A\times (\Omega\setminus A))+\nu(A)=\lim_{k\to+\infty}\Big(F(u_k)-\int_\Omega|\nabla u_k(x)|^2 \,dx\Big),
\\
&&2\mu(B\times (\Omega\setminus B))+\nu(B)=\lim_{k\to+\infty}\Big(F(v_k)-\int_\Omega|\nabla v_k(x)|^2 \,dx\Big),
\\
&&2\mu((A\cup B)\times (\Omega\setminus (A\cup B)))+\nu(A\cup B)\\
&&\hskip4cm=\lim_{k\to+\infty}\Big(F(u_k+v_k)-\int_\Omega|\nabla u_k(x)+\nabla v_k(x)|^2 \,dx\Big),
\end{eqnarray*}
where $u_k, v_k$ are sequences in $C^\infty_c(\Omega)$ with $0\le u_k\le 1_A$ and $0\le v_k\le 1_B$, such that $u_k(x)\to 1_A(x)$ and $v_k(x)\to 1_B(x)$ for all $x\in \Omega$. Summing up the first two equations above and subtracting the third one we have
\begin{eqnarray*}
\mu(A\times B)&=& {1\over 2} \Big(\mu(A\times (\Omega\setminus A))+\mu(B\times (\Omega\setminus B))-\mu((A\cup B)\times (\Omega\setminus (A\cup B)))\Big) \\
&=& \Phi_F(A,B),
\end{eqnarray*}
where
$$
\Phi_F(A,B):={1\over 4}\lim_{k\to+\infty}\big(F(u_k)+F(v_k)-F(u_k+v_k)\big).
$$
This shows that if $\mu_1$ and $\mu_2$ are symmetric bounded Borel measures and  satisfy \eqref{repr-F} 
then 
\begin{equation}\label{eqmu}
\mu_1(A\times B)=\mu_2(A\times B)
\end{equation}
for all disjoint open subsets $A$ and $B$ of $\Omega$. This property can be extended first to disjoint compact subsets of $\Omega$ and then to disjoint Borel subsets of $\Omega$. 

To extend this equality to arbitrary Borel subsets of $\Omega$ we fix $\eta>0$ and write
$$
A=\bigcup_{i=1}^{m_\eta} A_i,\qquad B=\bigcup_{j=1}^{n_\eta} B_j,
$$
where $A_i$ and $B_j$ are Borel partitions of $A$ and $B$, respectively, with diam$(A_i)<{\eta\over 2}$ and diam$(B_j)<{\eta\over 2}$ for all $i,j$. Setting $\mathcal D_\eta=\{(i,j): A_i\cap B_j\neq\emptyset\}$ and observing that $A\times B=\bigcup_{(i,j)}( A_i\times B_j)$ and that $\bigcup_{(i,j)\in \mathcal D_\eta}(A_i\times B_j)\subset\Delta_\eta$, by \eqref{eqmu} we obtain that 
\begin{eqnarray*}
|\mu_1(A\times B)-\mu_2(A\times B)|&=& \Big|\mu_1\Big(\bigcup_{(i,j)\in \mathcal D_\eta}(A_i\times B_j)\Big)-
 \mu_2\Big(\bigcup_{(i,j)\in \mathcal D_\eta}(A_i\times B_j)\Big)\Big|\\
 &\le& \mu_1((\Omega\times\Omega)\cap \Delta_\eta)+ \mu_2((\Omega\times\Omega)\cap \Delta_\eta).
\end{eqnarray*}
Therefore, if $\mu_1$ and $\mu_2$ also satisfy \eqref{stimadiag}, then, by letting $\eta\to 0$ we obtain that $\mu_1(A\times B)=\mu_2(A\times B)$ for all pairs of Borel sets, and hence that $\mu_1=\mu_2$. Finally, by \eqref{repr-F} we deduce that for every $u\in H^1_0(\Omega)$ the integral $\int_\Omega|u(x)|^2d\nu(x)$ is uniquely determined by $F$, which gives the uniqueness of such a $\nu$, and concludes the proof of (a).
\end{proof}

In the following theorem we use the classical notion of capacity, and $\widetilde u(x)$ denotes the precise representative of a function $u\in H^1_0(\Omega)$, which is defined up to sets of zero capacity (see \cite[Sections 4.7 and 4.8]{EG}). A similar result can be proved using the intrinsic capacity of the Dirichlet form $F$ and the corresponding precise representatives (see \cite[Theorem 4.5.2]{Fu}).

\begin{theorem}\label{capacity}
 Let $F\colon H^1_0(\Omega)\to [0,+\infty)$ and let  $\mu$ and $\nu$ be two bounded positive Radon measures on $\Omega\times\Omega$ and  $\Omega$, respectively, that satisfy {\rm(a)}, {\rm(b)}, and {\rm(c)} of Theorem~{\rm\ref{Integral rep}}. Suppose also that $F$ be  lower semicontinuous in the weak topology of $H^1_0(\Omega)$ and that there exists $M>0$ such that \eqref{stima6.1} holds for all $u\in C^\infty_c(\Omega)$. Then

{\rm(a)} if $B\subset\Omega$ is a Borel set with zero capacity, then $\mu(B\times\Omega)=\nu(B)=0$;

{\rm(b)} for every $u\in H^1_0(\Omega)\cap L^\infty(\Omega)$
\begin{equation}\label{Ftildeu}
F(u)=\int_{\Omega\times\Omega}|\widetilde u(x)-\widetilde u(y)|^2d\mu(x,y)+\int_\Omega|\widetilde u(x)|^2 d\nu(x)+\int_\Omega|\nabla u(x)|^2 \,dx;
\end{equation}

{\rm(c)} if, in addition, $F(\Psi_m(u))\le F(u)$ for all $u\in H^1_0(\Omega)$ and $m\in\mathbb N$, where $\Psi_m(t)= (m\wedge t)\vee (-m)$, then \eqref{Ftildeu} holds for every $u\in H^1_0(\Omega)$.
\end{theorem}

\begin{proof} We observe that, by the strong continuity of $u\mapsto\int_\Omega|\nabla u(x)|^2\dx$,  \begin{equation}\label{sci}
u\mapsto \int_{\Omega\times\Omega}|u(x)-u(y)|^2d\mu(x,y)+\int_\Omega|u(x)|^2 d\nu(x)\ \hbox{ is lower semicontinuous} \end{equation}
 on $C^\infty_c(\Omega)$ with respect to the strong topology of $H^1_0(\Omega)$.

Let $K$ be a compact subset of $\Omega$ with zero capacity. We now prove that 
\begin{eqnarray}\label{K-cap}
\mu(K\times(\Omega\setminus K))=0\quad\hbox{ and }\quad
\nu(K)=0.
\end{eqnarray}

Given $\eta>0$ let $U$ be an open set such that $K\subset U\subset \Omega$ and 
\begin{eqnarray}\label{U-cap}
\mu((U\setminus K)\times \Omega)\le \eta\quad\hbox{ and }\quad
\nu(U\setminus K)\le \eta,
\end{eqnarray}
and let $w\in C^\infty_c(\Omega)$ be such that $0\le w\le 1$ on the whole $\Omega$, $w=1$ in a neighbourhood of $K$, and $w\le {1\over 4}$ on $\Omega\setminus U$. Since $K$ has zero capacity there exist a sequence $u_k\in C^\infty_c(\Omega)$ converging to $w$ strongly in $H^1_0(\Omega)$, such that $0\le u_k\le 1$ on the whole $\Omega$, $u_k=0$ on $K$ and $u_k=w$ on $\Omega\setminus U$. Then we have
\begin{eqnarray*}
&&\hskip-.5cm\int_{\Omega\times\Omega} |u_k(x)-u_k(y)|^2d\mu(x,y)+ \int_\Omega |u_k(x)|^2d\nu(x) \\
&\le &2\int_{K\times (\Omega\setminus U)}|w(y)|^2d\mu(x,y) + \int_{(\Omega\setminus U)\times(\Omega\setminus U)} |w(x)-w(y)|^2d\mu(x,y)\\
&&+\mu((U\setminus K)\times (U\setminus K))+ 2\mu((U\setminus K)\times (\Omega\setminus U))
+\int_{\Omega\setminus U} |w(x)|^2d\nu(x) +\nu(U\setminus K)\\
&\le &{1\over 8}\mu(K\times (\Omega\setminus U)) + \int_{(\Omega\setminus U)\times(\Omega\setminus U)} |w(x)-w(y)|^2d\mu(x,y)\\
&&+
\int_{\Omega\setminus U} |w(x)|^2d\nu(x) +3\eta,
\end{eqnarray*}
while
\begin{eqnarray*}
&&\hskip-.5cm\int_{\Omega\times\Omega} |w(x)-w(y)|^2d\mu(x,y)+ \int_\Omega |w(x)|^2d\nu(x) \\
&\ge & 2\int_{K\times (\Omega\setminus U)}|1-w(y)|^2d\mu(x,y) +\int_{(\Omega\setminus U)\times(\Omega\setminus U)} |w(x)-w(y)|^2d\mu(x,y)\\
&& +\nu(K)+\int_{\Omega\setminus U} |w(x)|^2d\nu(x).\end{eqnarray*}
Hence, noting that $1-w(y)\ge {3\over 4}$ if $y\in \Omega\setminus U$, from the convergence of $u_k$ to $u$ and \eqref{sci} we obtain
\begin{eqnarray*}
{9\over 8}\mu(K\times (\Omega\setminus U))+\nu(K)\le {1\over 8}\mu(K\times (\Omega\setminus U)) +3\eta.
\end{eqnarray*} 
Taking into account that  $\mu(K\times (U\setminus K))\le \eta$ thanks to \eqref{U-cap} and the symmetry of $\mu$ we obtain
\begin{eqnarray*}
\mu(K\times (\Omega\setminus K))+\nu(K)\le 4\eta,
\end{eqnarray*} 
and \eqref{K-cap} is proved by the arbitrariness of $\eta$.

We now claim that
\begin{eqnarray}\label{K-K}
\mu(K\times K)=0,
\end{eqnarray} 
Given $\eta>0$ we can find a finite number of compact sets $K_i$ such that 
$K= \bigcup_i K_i$ and diam$K_i \le {\eta\over 2}$. Since 
\begin{eqnarray*}
K\times K\subset \bigcup_{i,j} (K_i\times K_j) = \bigcup_{(i,j)\in\mathcal D} (K_i\times K_j)\cup  \bigcup_{(i,j)\not\in\mathcal D} (K_i\times K_j),
\end{eqnarray*} 
where $\mathcal D=\{(i,j): K_i\cap K_j\neq\emptyset\}$, we have
\begin{eqnarray*}
\mu(K\times K)\le \mu\Big(\bigcup_{(i,j)\in\mathcal D} (K_i\times K_j)\Big) +\sum_{(i,j)\not\in\mathcal D}\mu( K_i\times K_j).
\end{eqnarray*} 
Since $K_j\subset \Omega\setminus K_i$ if $(i,j)\not\in\mathcal D$, by \eqref{K-cap} applied to $K_i$ the terms in the last sum are all zero. On the other hand $\bigcup_{(i,j)\in\mathcal D}( K_i\times K_j)\subset \Delta_\eta$, where $\Delta_\eta$ is  defined in \eqref{diag-eta}, so that $\mu(K\times K)\le \mu((\Omega\times \Omega)\cap \Delta_\eta)$. Since $\mu((\Omega\times \Omega)\cap \Delta)=0$ and $\mu$ is finite we obtain \eqref{K-K} by letting $\eta\to 0$.

Finally,  \eqref{K-cap} and \eqref{K-K} give that $\mu(K\times \Omega)=
\nu(K)=0$ for any $K$ compact set with zero capacity. Claim (a) is then obtain by approximation of $B$ with compact  sets contained in~$B$.

\smallskip
In order to prove claim (b), by proceeding as in the proof of Proposition \ref{prop-BD}(e) for all $u\in H^1_0(\Omega)\cap L^\infty(\Omega)$ we have a sequence $u_k\in C^\infty_c(\Omega)$ converging strongly to $u$ in $H^1_0(\Omega)$,  such that $\|u_k\|_{ L^\infty(\Omega)}\le \|u\|_{ L^\infty(\Omega)}$ and $u_k$ converge weakly to $u$  with respect to the Hilbert structure induced by the norm defined in \eqref{normF}.
By Mazur's theorem we obtain a new sequence $v_k\in C^\infty_c(\Omega)$ converging strongly to $u$ in $H^1_0(\Omega)$,  such that $\|v_k\|_{ L^\infty(\Omega)}\le \|u\|_{ L^\infty(\Omega)}$ and $v_k$ converge strongly to $u$ both in $H^1_0(\Omega)$ and with respect to the Hilbert structure induced by the norm defined in \eqref{normF}. In particular $F(v_k)\to F(u)$ and, upon passing to a subsequence, $v_k\to \widetilde u$ quasi-everywhere (in the sense of capacity). Together with the uniform bound, this implies that
$$
\lim_{k\to+\infty}\int_{\Omega\times\Omega}|v_k(x)-v_k(y)|^2d\mu(x,y)= \int_{\Omega\times\Omega}|\widetilde u(x)-\widetilde u(y)|^2d\mu(x,y)
$$
$$
\lim_{k\to+\infty}\int_{\Omega}|v_k(x)|^2d\nu(x)= \int_{\Omega}|\widetilde u(x)|^2d\nu(x),
$$
and claim (b) follows.

In order to prove (c), let $G\colon H^1_0(\Omega)\to [0,+\infty]$ be defined by the right-hand side in \eqref{Ftildeu}, and let $u\in H^1_0(\Omega)$. Let $u_m=\Phi_m(u)$ and note that $\lim\limits_{m\to+\infty} G(u_m)= G(u)$ by the Monotone Convergence . Since $F(u_m)= G(u_m)$ to conclude the proof it is enough to note that $\lim\limits_{m\to+\infty}F(u_m)= F(u)$, which follows from the hypothesis on truncations and the lower semicontinuity of $F$. 
\end{proof}

\begin{remark}\rm
Note that equality \eqref{Ftildeu} may not hold  on the whole $H^1_0(\Omega)$ if the additional assumption in (c) is dropped. For instance,  if $G$ is defined by the right-hand side in \eqref{Ftildeu}, and $F$ is defined as equal to $G$ except on a single $u_0\in H^1_0(\Omega)\setminus L^\infty(\Omega)$, where we set $F(u_0)=0$, then $F$ is lower semicontinuous, $F$ and $G$ are equal on $H^1_0(\Omega)\cap L^\infty(\Omega)$, but equality does not hold in the whole $H^1_0(\Omega)$.
\end{remark}

\begin{proposition}\label{vanish} In addition to the hypotheses of Theorem {\rm\ref{Integral rep}}, suppose that $a_k$ satisfies the following condition: for every $\e>0$ there exists a compact set $K_\e\subset \Omega$ such that
\begin{equation}\label{comp-ak}
\int_{(\Omega\times\Omega)\setminus (K_\e\times K_\e)} a_k(x,y)\dxy\le \e \hbox{ for every }k\in\mathbb N.
\end{equation}
Then the measure $\nu$ in Theorem {\rm\ref{Integral rep}} is the null measure.
\end{proposition}

\begin{proof}
Let $ u\in C^\infty_c(\Omega)$ such that $0\le u\le 1$ and $u=1$ on $K_\e$. 
Then, by using $u$ as test function in the $\Gamma$-limit we have
\begin{eqnarray*}
F(u)&\le &\liminf_{k\to+\infty} \int_{(\Omega\times\Omega)\setminus (K_\e\times K_\e)}|u(x)-u(y)|^2\dxy
+\int_\Omega|\nabla u(x)|^2\dx
\\
&\le& \e+\int_\Omega|\nabla u(x)|^2\dx.
\end{eqnarray*}
Since $F(u)=F^n(u)+F^\ell(u)+ \int_\Omega|\nabla u(x)|^2\dx$ we conclude that
$F^\ell(u)\le \e$ for all such $u$.

We now fix a compact $K$ in $\Omega$ and for each $\e>0$ take $u_\e\in C^\infty_c(\Omega)$ with $0\le u_\e\le 1$ and $u_\e=1$ on $K\cup K_\e$. By the estimate above we have
$$
\nu(K)\le \int_{\Omega}|u_\e(x)|^2d\nu(x)= F^\ell(u_\e)\le \e .
$$
By the arbitrariness of $\e$ we obtain that $\nu(K)=0$ for all $K$ compact of $\Omega$, which proves the claim.
\end{proof}

The following corollary improves the conclusions of Theorem \ref{Integral rep} in light of Theorems \ref{stimeme} and \ref{capacity} and of Proposition \ref{vanish}.

\begin{corollary}\label{cor}
Let $F_k$ be given by \eqref{efffek}, with $a_k$ satisfying \eqref{estimaak}. Suppose that $F_k$ $\Gamma$-converges to $F$ with respect to the weak topology in $H^1_0(\Omega)$. 
Then there exist two positive finite Radon measures $\mu$ and $\nu$  on $\Omega\times\Omega$ and  $\Omega$, respectively, such that

 {\rm(a)} $\mu$  is symmetric and $\mu((\Omega\times\Omega)\setminus \Delta)=0$;

{\rm(b)}  $\mu(B\times\Omega)=\nu(B)=0$ for all Borel sets $B\subset\Omega$with zero capacity;

{\rm(c)} for every $u\in H^1_0(\Omega)$
\begin{equation}\label{Gtildeu}
F(u)=\int_{\Omega\times\Omega}|\widetilde u(x)-\widetilde u(y)|^2d\mu(x,y)+\int_\Omega|\widetilde u(x)|^2 d\nu(x)+\int_\Omega|\nabla u(x)|^2 \,dx;
\end{equation}
If, in addition, $a_k$ satisfies property \eqref{comp-ak}, then $\nu=0$.
\end{corollary}

\begin{proof} This corollary is an immediate consequence of Theorems \ref{Integral rep}, \ref{stimeme}, and \ref{capacity}, noting that the hypothesis of Theorem \ref{capacity}(c) is satisfied thanks Proposition \ref{prop-BD}(d). The last statement follows from Proposition \ref{vanish}.
\end{proof} 

\begin{remark}[extension to general double integrals]\rm The conclusions of Corollary \ref{cor} remain valid if we consider the functionals defined by 
\begin{equation}
F_k(u)=\int_{\Omega\times\Omega}|\widetilde u(x)-\widetilde u(y)|^2d\mu_k(x,y)+\int_\Omega|\nabla u(x)|^2 \,dx
\end{equation}
for every $u\in H^1_0(\Omega)$, with condition \eqref{estimaak} substituted by
\begin{equation}\label{estimamuk} 
\mu_k(\Omega\times\Omega)\le M \hbox{ for every }k\in\mathbb N,
\end{equation}
and \eqref{comp-ak}  substituted by
\begin{equation}\label{comp-muk}
\mu_k((\Omega\times\Omega)\setminus (K_\e\times K_\e))\le \e \hbox{ for every }k\in\mathbb N.
\end{equation}
\end{remark}

\section{The counterexample}\label{Sec3}
We fix $p\in(1,+\infty)$. For simplicity of notation we suppose $0\in\Omega$, and let  $B_r$ be the ball of centre $0$ and radius $r$. We fix a sequence of positive numbers $\e_k$ converging to $0$, and define the functionals $F_k\colon  W^{1,p}_0(\Omega)\to \mathbb R$ by setting
\begin{equation}\label{effe-k}
F_k(u)={1\over  |B_{\e_k}|}\int_\Omega \int_{B_{\e_k} } |u(x)-u(y)|^p\dy\dx+ \int_\Omega |\nabla u(x)|^p\dx
\end{equation}
for every $u\in W^{1,p}_0(\Omega)$.

\begin{definition} For all  $u\in L^p(\Omega)$ we define $m_p(u)$ as the unique minimum point of
$$
t\mapsto \int_\Omega |u(x)-t|^p\,dx.
$$

\end{definition}

\begin{lemma}\label{lemma} The map $m_p\colon  L^p(\Omega)\to \mathbb R$ is continuous.\end{lemma}

\begin{proof}
Let $u_k\to u$ in $L^p(\Omega)$. Then the sequence $m_p(u_k)$ is bounded since
$$
|m_p(u_k)|\le \|u_k-m_p(u_k)\|_{L^p(\Omega)}+\|u_k\|_{L^p(\Omega)}\le 2\|u_k\|_{L^p(\Omega)}.
$$
We can suppose, upon subsequences, that $m_p(u_k)\to t_0$. With fixed $t\in\mathbb R$, we can pass to the limit in the inequalities
$$
\int_\Omega |u_k(x)-m_p(u_k)|^p\dx\le \int_\Omega |u_k(x)-t|^p\dx,
$$
and obtain that
$$
\int_\Omega |u(x)-t_0|^p\dx\le \int_\Omega |u(x)-t|^p\dx,
$$
which concludes the proof.
\end{proof}

\begin{theorem}\label{Gamma counterexample}
If $p\in (1,d]$ then the $\Gamma$-limit of $F_k$ with respect to the weak $W^{1,p}_0$-convergence is the functional $F$ defined by
\begin{equation}\label{effe}
F(u)=\int_\Omega| u(x)- m_p(u)|^p\dx+\int_\Omega|\nabla u(x)|^p \,dx
\end{equation}
for every $u\in W^{1,p}_0(\Omega)$.
\end{theorem}

\begin{proof} 
Let $u_k\wto u$ weakly  in $W^{1,p}_0(\Omega)$. Then also $u_k\to u$ strongly in $L^p(\Omega)$. Hence, using Jensen's inequality, the minimality of $m_p(u_k)$, and applying Lemma \ref{lemma}, we get
\begin{eqnarray*}
&&\hskip-2cm\liminf_{k\to+\infty} {1\over  |B_{\e_k}|}\int_\Omega \int_{B_{\e_k} } |u_k(x)-u_k(y)|^p\dy\dx
\\
&\ge&\liminf_{k\to+\infty} \int_{\Omega}\bigg|u_k(x)-{1\over |B_{\e_k}|}\int_{B_{\e_k}}u_k(y)\dy\bigg|^p dx 
\\
&\ge&\liminf_{k\to+\infty} \int_{\Omega}|u_k(x)-m_p(u_k)|^p dx\ge \int_{\Omega}|u(x)-m_p(u)|^p dx.
\end{eqnarray*}
Since the term $\int_\Omega|\nabla u(x)|^p\dx$ is lower semicontinuous, this proves the liminf inequality.

To prove the upper bound, we first construct a recovery sequence if $u=0$ in a neighbourhood of $0$.
In this case, let $v_k$ be the $p$-capacitary potential of $B_{\e_k}$ with respect to $\Omega$; that is, the minimizer of
$$
\min\Big\{\int_{\Omega}|\nabla v(x)|^p\,dx: v\in W^{1,p}_0(\Omega), v=1 \hbox { on }B_{\e_k}\Big\}.
$$
Since $p\le d$ it is known that $v_k\to 0$ in $W^{1,p}_0(\Omega)$  (see e.g.~\cite[Section 4.7]{EG}). We then set
$u_k=u+m_p(u)v_k$, and obtain
$$
\limsup_{k\to+\infty} F_k(u_k)= \limsup_{k\to+\infty} \Bigl(\int_\Omega |u_k(x)-m_p(u)|^p\dx+\int_\Omega |\nabla u_k(x)|^p\dx\Bigr) = F(u)
$$
Since $F$ is continuous in $W^{1,p}_0(\Omega)$ and the set of function $C^\infty_c(\Omega)$ 
which are $0$ in a neighbourhood of $0$ 
is dense in $W^{1,p}_0(\Omega)$,
the claim follows.
\end{proof}

\begin{remark}[$\Gamma$-limit in $W^{1,p}(\Omega)$]\rm
If $\Omega$ is a bounded open set with Lipschitz boundary, then the functionals defined by \eqref{effe-k} for $u\in W^{1,p}(\Omega)$ $\Gamma$-converge with respect to the $L^p(\Omega)$ convergence to the functional 
defined by \eqref{effe} for $u\in W^{1,p}(\Omega)$. Indeed, in the proof we only use the boundary condition to deduce the equi-coerciveness of the functionals, a property that is also assured by the regularity of $\partial\Omega$.
\end{remark}

We now want to show that $F$ cannot be represented in the form
$$
F(u)=\int_{\Omega\times\Omega}f(u(x),u(y))d\mu(x,y)+\int_\Omega g(u(x))d\nu(x)+\int_\Omega|\nabla u(x)|^p\dx
$$
for $u\in C^\infty_c(\Omega)$, where $f\colon \mathbb R^2\to\mathbb R$ and $g\colon \mathbb R\to \mathbb R$
are continuous functions, while $\mu$ and $\nu$ are two positive bounded Radon measures on $\Omega\times\Omega$ and  $\Omega$, respectively.
%
%
%
To that end we examine the two integrals with respect to $\mu$ and $\nu$ separately from the third one.

\begin{proposition}\label{prop1}  Let $f\colon \mathbb R^2\to\mathbb R$ and $g\colon \mathbb R\to \mathbb R$
be continuous functions, and let $\mu$ and $\nu$ be two positive bounded Radon measures on $\Omega\times\Omega$ and  $\Omega$, respectively. Suppose that
\begin{equation}\label{1}
\int_\Omega| u(x)- m_p(u)|^p\dx= \int_{\Omega\times\Omega}f(u(x),u(y))d\mu(x,y)+\int_\Omega g(u(x))d\nu(x)
\end{equation} holds for $u\in C^\infty_c(\Omega)$, then the same equality holds also for $u=1_A$, for all $A$ open of $\Omega$;
that is,
\begin{equation}\label{2}
\int_\Omega| 1_A(x)- m_p(1_A)|^p\dx= \int_{\Omega\times\Omega}f(1_A(x),1_A(y))d\mu(x,y)+\int_\Omega g(1_A(x))d\nu(x).
\end{equation}
\end{proposition}

\begin{proof} Preliminarily, note that, taking $u=0$ in \eqref{1}, we obtain
\begin{equation*}
0= f(0,0)\,\mu(\Omega\times\Omega)+ g(0)\,\nu(\Omega).
\end{equation*}
It is  then not restrictive to assume that 
\begin{equation}\label{zero}
f(0,0)=g(0)=0,
\end{equation}
up to substituting $f(s,t)$ with $f(s,t)-f(0,0)$ and $g(s)$ with $g(s)-g(0)$.

Let now $A$ be an open set relatively compact in $\Omega$, and let $u_k$ be a sequence in $C^\infty_c(\Omega)$ converging pointwise to $1_A$ and such that $0\le u_k(x)\le 1_A(x)$. By the convergence of $u_k$ to $1_A$ in $L^p(\Omega)$ and Lemma \ref{lemma} we have the convergence of the left-hand term in \eqref{1} to the corresponding term in \eqref{2}. As for the right-hand side of \eqref{1}, it suffices to apply the 
Dominated Convergence Theorem.
\end{proof}


%

\begin{remark}[computation of $m_p(u)$ for characteristic functions]
\rm  For any measurable set $A$ the constant  $m_p(1_A)$ is obtained by minimizing 
$$
 \int_\Omega| 1_A(x)- t|^p\dx= |A||1-t|^p+(|\Omega|-|A|) |t|^p.
$$
The minimal $t\in[0,1]$ is determined by 
$(|\Omega|-|A|) t^{p-1}=|A|(1-t)^{p-1}$;
that is, we have
$$
{1-t\over t}= \Bigl({|\Omega|-|A|\over|A|}\Bigr)^{1/(p-1)}
\quad\hbox{and  }\quad
m_p(1_A)= {|A|^{1/(p-1)}\over (|\Omega|-|A|)^{1/(p-1)}+|A|^{1/(p-1)}}.
$$
\end{remark}

\begin{remark}\rm
From the previous remark we have that \begin{equation}\label{33}
\int\limits_\Omega| 1_A(x)- m_p(1_A)|^p\dx= \Phi_p(|A|),\end{equation} where
\begin{equation}\label{fp}
\Phi_p(s):={ s(|\Omega|-s)^{p/(p-1)} +  (|\Omega|-s)s^{p/(p-1)}\over
((|\Omega|-s)^{1/(p-1)}+s^{1/(p-1)})^p}= { s(|\Omega|-s)\over
((|\Omega|-s)^{1/(p-1)}+s^{1/(p-1)})^{p-1}}.
\end{equation}
\end{remark}

The following proposition relates the function $\Phi_p$ defined in \eqref{fp} and the measure $\mu$.

\begin{proposition}\label{prop2} Under the assumptions of Proposition {\rm\ref{prop1}}  for all $A,B$ open sets in $\Omega$ with $A\cap B=\emptyset$ we have
\begin{equation}\label{3b}
\Phi_p(|A|)+\Phi_p(|B|)- \Phi_p(|A|+|B|)=C_f(\mu(A\times B) + \mu(B\times A)),
\end{equation}
where $C_f=f(1,0)+f(0,1)-f(1,1)$.
\end{proposition}

\begin{proof} From \eqref{2} and \eqref{33} we have 
\begin{eqnarray}\label{3}\nonumber
&&\Phi_p(|A|)= f(1,1)\mu(A\times A)+ f(1, 0)\mu(A\times(\Omega\setminus A))\\&&
\hskip4.5cm+f(0,1)\mu((\Omega\setminus A)\times A) +g(1)\nu(A),
\end{eqnarray}
and analogous formulas for $\Phi_p(|B|)$ and $\Phi_p(|A|+|B|)=\Phi_p(|A\cup B|)$,
from which the claim follows.
\end{proof}

\begin{proposition}\label{prop3} If there exists a bounded Radon measure $\mu$ on $\Omega\times\Omega$ such that \eqref{3b} holds, then $p=2$.\end{proposition}

\begin{proof}
Take $A_1$, $A_2$, and $B$ disjoint open subsets of $\Omega$, and let $s_1=|A_1|$,  $s_2=|A_2|$, and $t=|B|$.
From \eqref{3b} we then have
$$
\Phi_p(s_1+s_2)+\Phi_p(t)- \Phi_p(s_1+s_2+t)=C_f( \mu((A_1\cup A_2)\times B) + \mu(B\times (A_1\cup A_2)))
$$
$$=C_f(\mu(A_1\times B)+ \mu(B\times A_1) +\mu(A_2\times B)   + \mu(B\times A_2))
$$
$$=\Phi_p(s_1)+\Phi_p(t)- \Phi_p(s_1+t) +\Phi_p(s_2)+\Phi_p(t)- \Phi_p(s_2+t),
$$
or, equivalently, that for every fixed $t\in(0,|\Omega|)$ the function
$$
g(s)= \Phi_p(s)+\Phi_p(t)- \Phi_p(s+t)
$$
is additive on $(0,|\Omega|-t)$, which implies that there exists a constant $c_t$ such that $g(s)= c_t s$. 
In particular, taking into account the differentiability of $\Phi_p$, we have $$\Phi''_p(s)- \Phi''_p(s+t)=g''(s)= 0\hbox{ for all } s, t\in (0,|\Omega|) \hbox{ such that } s+t< |\Omega|.
$$
This implies that $\Phi''_p$ is constant, so that it equals a second-order polynomial $P$.

It is now convenient to write $\Phi_p(s)= s(|\Omega|-s) (h_p(s))^{1-p}$, where 
$$
h_p(s)= s^{1/(p-1)}+(|\Omega|-s)^{1/(p-1)}.
$$
Since $h_p(s)\neq0$ we have $P(s)=0$ if and only if $s=0$ or $s=|\Omega|$, so that $P(s) = \kappa s(|\Omega|-s)$ for some constant $\kappa$.
This implies that $h_p(s)=\kappa$ for every $s\in (0,|\Omega|)$ and then also for $s=0$ and $s=|\Omega|$ by continuity. In particular, this gives 
$$
|\Omega|^{1/(p-1)}=h_p(0)= h_p\Bigl({|\Omega|\over 2}\Bigr)= 2\Bigl({|\Omega|\over 2}\Bigr)^{1/(p-1)},
$$
which holds only if $p=2$.
\end{proof}
%
%
%
%
%
%
%
%
%
%
%

Combining the previous results, we are now in a position to prove that $F$ cannot be represented in an integral form when $1<p\le d$ and $p\neq2$.
\begin{corollary}\label{cor: main} Let $1<p\le d$ and let $F$ be the  $\Gamma$-limit, with respect to the weak $W^{1,p}_0$-convergence, of the sequence $F_k$ defined by \eqref{effe-k}.
Suppose that there exist two real valued continuous functions  $f$ and $g$, defined on $\mathbb R^2$ and $ \mathbb R$, 
 and two positive bounded Radon measures on $\Omega\times\Omega$ and  $\Omega$, respectively, such that
\begin{equation}\label{rappresentazione}
F(u)=\int_{\Omega\times\Omega}f(u(x),u(y))d\mu(x,y)+\int_\Omega g(u(x))d\nu(x)+\int_\Omega|\nabla u(x)|^p\dx
\end{equation}
for every $u\in C^\infty_c(\Omega)$. Then $p=2$ and in this case we have $\mu={1\over 2|\Omega|}\mathcal L^{2d}$ and $\nu=0$, while  $f(s,t)=|s-t|^2$ for every $s, t\in\mathbb R$.
\end{corollary}

\begin{proof}
By Theorem \ref{Gamma counterexample} the functional $F$ is given by \eqref{effe}. By \eqref{rappresentazione} this implies that the assumptions of Proposition \ref{prop1} are satisfied,
and by Proposition \ref{prop2} we can apply Proposition \ref{prop3}, which gives $p=2$. The explicit form of of $f$, $\mu$, and $\nu$ follows from \eqref{intro10} and \eqref{effe}. 
\end{proof}

\bigskip

\noindent \textsc{Acknowledgements.}
 This paper is based on work supported by the National Research Project (PRIN  2017BTM7SN) 
 "Variational Methods for Stationary and Evolution Problems with Singularities and 
 Interfaces", funded by the Italian Ministry of University and Research. 
The authors are members of the Gruppo Nazionale per 
l'Analisi Matematica, la Probabilit\`a e le loro Applicazioni (GNAMPA) of the 
Istituto Nazionale di Alta Matematica (INdAM).

\end{document}